\mag=\magstep1

\documentclass[10pt,a4paper,dviout]{amsart} 
\usepackage{amssymb,latexsym} 
\usepackage{here} 
\usepackage{color} 
\usepackage{comment}
\usepackage{datetime}
\usepackage{comment}
\usepackage{graphicx}

\input xy
\xyoption{all}

\usepackage{amsmath, amssymb, eucal, amscd, color, amsthm}

\def\Cal{\mathcal}

\def\<{\langle}
\def\>{\rangle}

\def\Im{\operatorname{Im}}

\def\BQ{\overline{\bold Q}}
\def\be{\beta}
\def\al{\alpha}

\newtheorem{theorem}{Theorem}[section] 
\newtheorem{proposition}[theorem]{Proposition}

\newtheorem{lemma-definition}[theorem]{Lemma-Definition}

\makeatletter
  
  \@addtoreset{equation}{section}
 \makeatother

\usepackage{version}	
\begin{document}

\title{
Period integral of open Fermat surfaces \\
and special values of hypergeometric functions} 
\author{Tomohide Terasoma}

\begin{abstract}In the paper \cite{AOT}, we prove that the special values $\ _3F_2(1,1,a;b,c;1)$
of the hypergeometric function $\ _3F_2$ is a $\overline{\bold Q}$-linear combination of
$\log(\lambda)$ for $\lambda \in \overline{\bold Q}$ and $1$, if $a,b,c$ are rational numbers
satisfying a certain condition.
In \cite{A}, \cite{S}, the triples $(a,b,c)$ with this condition are completely classified
in relation with Hodge cycles on Fermat surfaces.
In this paper, we give an explicit expression of $\ _3F_2(1,1,a;b,c;1)$
which does not belong to the finite exceptional characters in the list of \cite{S}.
\end{abstract}
\maketitle
\setcounter{tocdepth}{1}

\markboth
{Period of open Fermat surfaces and hypergeometric function}
{Tomohide Terasoma}

\thispagestyle{empty}

\tableofcontents
\section{Introduction}
\subsection{Introduction and result of \cite{AOT}}
Let $p_1, \dots, p_5$ be real numbers such that
$p_4, p_4 \notin -1,-2,\cdots$ and $x$ be a complex number with $|x|<1$. 
We define the hypergeometric function 
$\ _3F_2(p_1,p_2,p_3;p_4,p_5;x)=F(p_1,p_2,p_3;p_4,p_5;x)$
by the series (\cite{E})
$$
F(p_1,p_2,p_3;p_4,p_5;x)=\sum_{k=0}^{\infty}\dfrac
{(p_1)_k(p_2)_k(p_3)_k}
{(p_4)_k(p_5)_k k!}x^k,
$$ 
where $(p)_k$ ($k=0,1,\cdots$) is the Pochhammer symbol defined by
$$
(p)_k=p(p+1)\cdots (p+k-1)=\dfrac{\Gamma(p+k)}{\Gamma(p)}.
$$
For a positive integer $m$ and 
rational numbers $\alpha_0, \alpha_1, \alpha_2, \alpha_3 \in \dfrac{1}{m}\bold Z-\bold Z$,
the limit
\begin{equation}
\label{limit and special value}
F(1,1,\alpha_1+\alpha_2+\alpha_3;\alpha_1+\alpha_2;\alpha_1+\alpha_3;1)=\lim_{x\to 1}
F(1,1,\alpha_1+\alpha_2+\alpha_3;\alpha_1+\alpha_2;\alpha_1+\alpha_3;x)
\end{equation}
exists.
The fractional part $\<\alpha\>$ of a rational number $\alpha$ 
is a rational number characterized by $\alpha-\<\alpha\> \in \bold Z$ and $0\leq \<\alpha\><1$.
\begin{theorem}[see \cite{AOT}]
\label{Asakura-Otsubo}
If the parameters $\alpha_0,\dots, \alpha_3$ satisfy the condition
\begin{equation}
\label{Hodge cycles}
\<t\alpha_0\>+\<t\alpha_1\>+\<t\alpha_2\>+\<t\alpha_3\>=2, \text{ for all }t\in (\bold Z/m\bold Z)^{\times},
\end{equation}
then 
\begin{equation}
\label{special values}
F(1,1,\alpha_1+\alpha_2+\alpha_3;\alpha_1+\alpha_2;\alpha_1+\alpha_3;1)
\in \overline{\bold Q}+\overline{\bold Q}\log(\overline{\bold Q}^{\times}).
\end{equation}
Here, $\overline{\bold Q}+\overline{\bold Q}\log(\overline{\bold Q}^{\times})$
denotes the $\overline{\bold Q}$ linear hull of $\log(\lambda)$ with $\lambda \in \overline{\bold Q}^{\times}$
and $1$.
\end{theorem}

As is explained in the next subsection, the special value of hypergeometric function of the above type is related to
the extension class of mixed $K$-mixed Hodge structures
\begin{equation}
\label{first extension class}
0\to H^1_{Hg}(B,K)(\chi_{\alpha})\to
H^2_{Hg}(X_m,B,K)(\chi_{\alpha})\to
H^2_{Hg}(X_m,K)(\chi_{\alpha}) \to 0
\end{equation}
arising from the relative cohomologies of Fermat surface $X_m$ and its divisor $B$.
Theorem \ref{Asakura-Otsubo}
is a consequence of the fact that the extension class of 
(\ref{first extension class}) is actually 
a Hodge realization of a mixed Tate motives if the condition (\ref{Hodge cycles}) is satisfied.
In the proof of Theorem \label{Asakura-Otsubo} \ref{Asakura-Otsubo} in \cite{AOT}, 
we use Lefschetz-Hodge theorem
and the fact that an algebraic cycles on algebraic variety defined over $\overline{\bold Q}$
are defined on $\overline{\bold Q}$. Thus the method in \cite{AOT} does not give an closed formulas for
the special value (\ref{limit and special value}).
In this paper, we give an explicit formula for the above theorem
for non-exceptional cases (see Theorem \ref{Shioda's classification}
 for the definition of
exceptional cases).

Let us explain the outline of the paper. In Section 1, we recall 
the result of \cite{AOT} by introducing period integral of an open
Fermat surfaces, or dually period for the relative cohomology
Fermat surfaces with their divisors. Let $\chi$ be a character of a group $G_m$ acting
on $X_m$.
If the $\chi$-part of the cohomology
of Fermat surface is generated by algebraic cycles, the corresponding
periods can be written using the logarithmic function evaluated at
algebraic numbers. We recall the result by Aoki and Shioda on
algebraic cycles on Fermat surfaces. We also recall that
except for finitely many characters, they are obtained by
four types of characters. We give an explicit formula belonging to four types.

In Section 2, we explain that the existence of algebraic cycles
implies the exactness for certain rational differential forms.
The equation of algebraic cycles are key to find differential 
forms which bound the given differential forms.
Actually, if one find such differential form, the story is
independent of the existence of algebraic cycles.

From Section 3 to Section 5, we compute the integral and
show that the integral expressing the extension class 
actually are expressed by simpler integration, which
gives the explicit expression by logarithmic functions.
Here we use Stokes formula for the variety obtained by blowing up
$\bold C^2$.

In the first proof of \cite{AOT}, we used regulator maps for a symbols.
If the symbol can be written explicitly, we also have a closed formula of
the special value (\ref{limit and special value}).
In \cite{AY}, they obtained explicit formulas for the special values 
for some cases using this method. The one form obtained
in Section 2 gives a key to find an explicit expression of
related $K$-group via symbols. This topics will be left to a 
future research.

\vskip 0.1in
{\bf Acknowledgement}
This paper was first considered as a continuation of the paper
\cite{AOT}.
The author express his acknowledgement to M. Asakura and N. Otsubo
for discussion and private communication.
He also thank them to let him know the relationship with
regulator maps in $K$-theory and hypergeometric functions, which
is a driving force to let him go into this subject.

\subsection{Algebraic cycles on Fermat surfaces}
A proof of the above theorem is based on the fact that the value
(\ref{special values})
is considered as a period integral of
certain relative cohomology of Fermat surface.
We recall some properties of cohomologies of Fermat surfaces.
Let $X_m$ be an affine Fermat surfaces defined by
$$
X_m=Spec(\BQ[u,v,w]/(u^m+v^m-1-w^m))
$$
and $\mu_m$ be the group of $m$-th roots of unities in $\bold C^{\times}$
and set $K=\bold Q(\mu_m)$.
Then the group
$G_m=(\mu_m)^3$ acts on $X_m$ by
\begin{equation}
\label{Fermat abel action} 
\rho(g_1,g_2,g_3):(u,v,w)\mapsto (g_1u,g_3v,g_3w)
\end{equation}
for $(g_1,g_2,g_3) \in G_m$.
For $\bold a=(\alpha_0,\alpha_1,\alpha_2,\alpha_3)\in \dfrac{1}{m}\bold Z^4$ with
$0<\alpha_i<1$, $\sum_i\alpha_i\in \bold Z$, we define a character 
$\chi=\chi_{\bold a}$
by 
\begin{equation}
\label{Kummer char on Fermat}
\chi_{\bold a}(g_1,g_2,g_3)=g_1^{a_1}g_2^{a_2}g_3^{a_3}
\end{equation}
where $a_i=m\alpha_i$.
The $\chi_{\bold a}$-part of the singular cohomology $H^2_B(X_m,K)$ and algebraic de Rham cohomology
$H^2_{dR}(X_m/\overline{\bold Q})$ over $\overline{\bold Q}$ of $X_m$ 
are denoted by
$H^2_B(X_m,K)(\chi_{\bold a})$
and
$H^2_{dR}(X_m/\overline{\bold Q})(\chi_{\bold a})$, respectively.
Then we have
$$
\dim(H^2_{dR}(X_m/\overline{\bold Q})(\chi_{\bold a}))=
\dim(H^2_B(X_m,K)(\chi_{\bold a}))=1.
$$
The condition (\ref{Hodge cycles}) is
equivalent to the following condition (see \cite{S}):
\begin{equation}
\label{Hodge cycles 2}
\text{The space }H^2_B(X_m,K)(\chi_{\bold a}) \text{ is generated by Hodge cycles.}
\end{equation}
The complete classification of the set of indices $(a_0,a_1,a_2,a_3)$
satisfying the condition (\ref{Hodge cycles 2})
 is conjectured in \cite{S} and proved by \cite{A}.
\begin{theorem}[\cite{S}, \cite{A}]
\label{Shioda's classification}
Let $\alpha_0,\alpha_1,\alpha_2,\alpha_3$ be elements in
$\bold Q\cap (0,1)$ with $\sum_{i=0}^3\alpha_i\in \bold Z$
satisfying the condition (\ref{Hodge cycles}) and $m$ be the 
common denominator. Then one of the following holds.
\begin{enumerate}
 \item 
There exists and element $\alpha, \beta \in \bold Q$ such that
$(\alpha_0,\alpha_1,\alpha_2,\alpha_3)$ is equal to 
$(\alpha,-\alpha,\beta,-\beta)$ up to a permutation.
 \item 
There exists and element $\alpha \in \bold Q$ such that
$(\alpha_0,\alpha_1,\alpha_2,\alpha_3)$ is equal to 
one of the following up to a permutation.
\begin{enumerate}
 \item 
$(2\alpha,1-\alpha,-\alpha+\dfrac{1}{2},\dfrac{1}{2})$
\item
$(3\alpha,1-\alpha,-\alpha+\dfrac{1}{3},-\alpha+\dfrac{2}{3})$
\item
$(4\alpha,1-2\alpha,-\alpha+\dfrac{1}{4},-\alpha+\dfrac{3}{4})$
\end{enumerate}
\item
$m\leq 180$ and does not satisfies (1) and (2).
\end{enumerate}
In particular there exist only finitely many vectors $(\alpha_0, \dots, \alpha_3)$
belonging to the case (3).
\end{theorem}
The characters listed in (3) of Theorem \ref{Shioda's classification}, 
which is not appeared in (1), (2) is
called the exceptional characters.
There are $101$ Galois orbits in the exceptional characters 
and the list of orbits is given in Appendix (Compare for the list in \cite{S}).

\subsection{Extension of mixed Tate motives}
\subsubsection{Tate Hodge structure}
Let $m\geq 2$ be an integer and $K$ be the field generated by $\mu_m=\<\bold e(1/m)\>$ 
over $\bold Q$.
A mixed $K$-Hodge structure $V_{Hg}$ consists of triple
$$
V_{Hg}=((V_{dR}, F,\overline{F}, W),(V_B,W),c)
$$
of filtered vector space over $\bold C, K$ and a comparison isomorphism 
$$
c:V_B\otimes_K\bold C\xrightarrow{\simeq}V_{dR}
$$
compatible with $W$ such that $(F,\overline{F}, W)$ is oposit filtration defined in 
\cite{D}. For mixed $K$-Hodge structures $V_{1,Hg}, V_{2,Hg}$, a homomorphisms
from $\varphi:V_{1, Hg}\to V_{1, Hg}$ is a pair of homomorphisms 
$\varphi_B:V_{1,B} \to V_{2,B}$ and
$\varphi_{dR}:V_{1,dR} \to V_{2,dR}$ 
preserving comparison isomorphisms $c$ and filtrations $W, F$.
By proposition in  \cite{D}, the category of mixed $K$-Hodge structures
whose morphisms are morphism of mixed $K$-Hodge structures becomes
an abelian category. For a mixed Hodge structure $V_{Hg}=(V_B, V_{dR}, c)$
over $\bold Q$, the tensor product $V_{Hg}\otimes_{\bold Q} K=(V_{B}\otimes_{\bold Q}K, V_{dR},c)$
becomes a mixed $K$-Hodge structure in a natural way.

Tate Hodge structure $\bold Q(1)_{Hg}$ is defined by the triple
$(\bold Q(1)_{dR},\bold Q(1)_{B},c)$, where $\bold Q(1)_{dR}$ and $\bold Q(1)_{B}$
are one dimensional vector spaces over $\bold C$ and $\bold Q$ generated by
$\xi_{dR}$ and $\xi_{B}$ with $c(\xi_{dR})=(2\pi i)\xi_B$.
We define $\bold Q(i)_{Hg}=(\bold Q(i)_{dR},\bold Q(i)_{B},c)$,
where
$$
\bold Q(i)_{\star}=
\begin{cases}
L_{\star}^{\otimes i} & (i\geq 0)
\\
Hom_{L_{\star}}(L_{\star}^{\otimes (-i)},L_{\star})& (i<0)
\end{cases}
$$
for $\star=dR, B$ and $L_{dR}=\bold C, L_{B}=\bold Q$.
The tensor product $K(i)_{Hg}=\bold Q(i)_{Hg}\otimes_{\bold Q}K$ 
of $\bold Q(i)_{Hg}$ is also called a Tate Hodge structure.
For $K(-1)_{Hg}=(K(-1)_{dR},K(-1)_{B},c)$, $K(-1)_{dR}$ and $K(-1)_B$ are
generated by $\xi^{-1}_{dR}$ and $\xi^{-1}_{B}$.

\subsubsection{Extension class}
Let 
$$
e_{Hg}:0 \to K \to E \xrightarrow{p} K(-1)\to 0,
$$ 
be an exact sequence of mixed $K$-Hodge structures.
Let $c:K(-1)_B\otimes_K \bold C\to K(-1)_{dR}$
be the comparison isomorphism and
$s_B$ and $s_{dR}$ be elements in $E_{B}$ such that
\begin{enumerate}
\item
$p(s_{B})$ is an element in $E_{B}$ such that $p(s_{B})=\xi^{-1}_B$,
\item
$s_{dR}$ is the inverse image of $\xi_{dR}^{-1}$ 
under the isomorphism
$F^1E_{dR} \to F^1K(-1)_{dR}=K(-1)_{dR}$.
\end{enumerate}
Then we have an isomorphism 
$$
cl:Ext^1_{MHS(K)}(K(-1),K)=\bold C/2\pi iK
$$
by setting 
$$
cl(e_{Hg})=2\pi ic(s_B)-s_{dR} \in 2\pi iK_{dR}=\bold C\xi^0_{dR} \text{ mod }2\pi iK_B
=2\pi iK\xi_{B}^0.
$$
The lifting satisfying the above condition (2) is called
the Hodge canonical lift.

\subsubsection{Extensions arising from relative cohomologies of Fermat surfaces}

Let $(X,B)$ a pair of varieties such that $B\subset X$. Let $H^i_B(X,B,K)$ and
$H^i_{dR}(X,B/\bold C)$
be the singular and de Rham cohomologies of the pair $(X,B)$ with
the coefficient in $K$ and de Rham cohomology with the coefficient in $\bold C$.
By the natural comparison map $c$, we have a mixed $K$-Hodge structure $(H_{B}(X,B,K), H_{dR}(X,B),c)$
Let $G \to Aut(X,B)$ be a group action of $G$ on $(X,B)$ and
$\chi$ be a character of $G$ with the value in $K^{\times}$.
Then $H_{Hg}(X,B,K)(\chi)=(H_{B}(X,B,K)(\chi), H_{dR}(X,B)(\chi),c)$ becomes a mixed $K$-Hodge structure.

We define a subvariety $B$ of Fermat sufrace $X_m$ by
$$
B:=\{u^m=1, v^m=w^m\}\cup \{v^m=1, u^m=w^m\}.
$$
The the variety $B$ is a union of affine lines.

The $\chi_{\bold a}$-part of the relative singular cohomology $H^2_B(X,B,K)$
and relative de Rham cohomology $H^2_{dR}(X,B/\BQ)$ of $(X,B)$ are denoted by
$H^2_B(X,B,K)(\chi_{\bold a})$ and
$H^2_B(X,B/\BQ)(\chi_{\bold a})$, respectively.
By the long exact sequence of relative cohomology, we have
an exact sequences:
\begin{align}
\label{exact cohomology Hodge}
&0\to H^1_B(B,K)(\chi_{\bold a})\to
H^2_B(X,B,K)(\chi_{\bold a})\to
H^2_B(X,K)(\chi_{\bold a}) \to 0, \\
\nonumber
&0\to H^1_{dR}(B/\BQ)(\chi_{\bold a})\to
H^2_{dR}(X,B/\BQ)(\chi_{\bold a})\to
H^2_{dR}(X/\BQ)(\chi_{\bold a}) \to 0.
\end{align}
These sequences are compatible with the comparison maps
$$
c:H^2_B(X,K)(\chi_{\bold a})\otimes_K\bold C \xrightarrow{\simeq}
H^2_{dR}(X/\BQ)(\chi_{\bold a})\otimes_{\BQ}\bold C, \quad \text{ etc. }
$$
We set 
$$
H^2_{Hg}(X,K)(\chi_{\bold a})=(
H^2_{B}(X,K)(\chi_{\bold a}),
H^2_{dR}(X,K)(\chi_{\bold a})\otimes_{\overline{\bold Q}} \bold C,c)
$$
Therefore the sequences (\ref{exact cohomology Hodge})
together with the comparison maps
defines a Yoneda extension class $e$ of 
mixed $K$-Hodge structures
$$
e_{Hg}(X,\chi_{\bold a})\in Ext^1_{MHS(K)}(H^2_{Hg}(X,K)(\chi_{\bold a}),H^1_{Hg}(B,K)(\chi_{\bold a})).
$$
Under the condition (\ref{Hodge cycles 2}),
we have isomorphisms of Hodge structures
\begin{equation}
\label{isomorphism as motives}
H^1_{Hg}(B,K)(\chi_{\bold a})\xrightarrow{\iota_1} K_{Hg}, 
\quad H^2_{Hg}(X,K)(\chi_{\bold a})\xrightarrow{\iota_2} K(-1)_{Hg}.
\end{equation}

To prove Theorem \ref{Asakura-Otsubo}, we use theory of 
mixed motives and its Hodge realization.
By the Lefschetz-Hodge theorem, these isomorphisms arise from algebraic correspondences.
Since the pair of varieties $(X,B)$ are defined over $\overline{\bold Q}$,
the above algebraic correspondence is defined over $\overline{\bold Q}$.
Let $MM(K)/L$ be the derived category of mixed motives over $L$
with the coefficients in $M$.
Therefore we have the following commutative diagram where the horizontal
arrow are induced by algebraic correspondence over $\overline{\bold Q}$.
$$
\begin{matrix}
 Ext^1_{MM(K)/\overline{\bold Q}}(h(X)_K(\chi_{\bold a}),h(E)_K(\chi_{\bold a})) &\to&
Ext^1_{MM(K)/\overline{\bold Q}}(K(-1),K)
\\
\downarrow & & \downarrow
\\
 Ext^1_{MM(K)/\bold C}(h(X)_K(\chi_{\bold a}),h(E)_K(\chi_{\bold a})) &\to&
Ext^1_{MM(K)/\bold C}(K(-1),K)
\\
\downarrow & & \downarrow
\\
 Ext^1_{MHS(K)}(h(X)_K(\chi_{\bold a}),h(E)_K(\chi_{\bold a})) &\to&
Ext^1_{MHS(K)}(K(-1),K)
\end{matrix}
$$
Since the Hodge realization map 
$$
Ext^1_{MM(K)/\overline{\bold Q}}(K(-1),K)\simeq \overline{\bold Q}^\times \otimes K
\xrightarrow{\rho} 
Ext^1_{MHS(K)}(K(-1)_{Hg},K_{Hg})=(\bold C/2\pi iK)
$$
is given by $e_{M}\mapsto \log(e_M)$,
we have
$$
cl(e_{Hg}(X,\chi_{\bold a}))\in K\log(\overline{\bold Q}^{\times})=\Im(\rho),
$$


\subsubsection{Extension class and relative periods}

The extension class $e_{Hg}$ can be computed by the period integral as follows:
Let $s_B$, $s_{dR}$ be
liftings of a common bases of $H^2_B(X,K)(\chi_{\bold a})$
$H^2_{dR}(X,\BQ)(\chi_{\bold a})$ satisfying the conditions
(1)-(2) in the last subsection,
and $b_{dR}$ be a base of $H^1_{dR}(B/\BQ)(\chi_{\bold a})$.
Then we have $2\pi i c(s_B)=s_{dR}+cl(e_{Hg})b_{dR}$,
where $e_{Hg}=e_{Hg}(X,\chi_{\bold a})$.

We consider the Betti part of the dual of the exact sequences 
(\ref{exact cohomology Hodge}).
\begin{align*}
&0\to 
H_2^B(X,K)(\chi_{-\bold a}) \to  
H_2^B(X,B,K)(\chi_{-\bold a})\xrightarrow{\partial}
H_1^B(B,K)(\chi_{-\bold a})\to 0
\end{align*}
Let $\gamma$ be a $K$-base of 
$H_1^B(B,K)(\chi_{-\bold a})$
and $\Gamma$ be an element in
$H_2^B(X,B,K)(\chi_{-\bold a})$
such that $\partial\Gamma=\gamma$.
The cup product induces a map
$$
(*,*):H_2^B(X,B,K)(\chi_{-\bold a})\times
H^2_B(X,B,K)(\chi_{\bold a})\to K.
$$
Then
\begin{align*}
2\pi i(\Gamma,s_B)
=&2\pi i(c(\Gamma),c(s_B))
=(c(\Gamma),s_{dR})+cl(e_{Hg})(c(\Gamma),b_{dR})
\\
=&(c(\Gamma),s_{dR})+cl(e_{Hg})(c(\gamma),b_{dR}).
\end{align*}
Since $(\Gamma,s_B)\in K$, $(c(\gamma),b_{dR})\in \BQ$,
we have
$$
(c(\Gamma),s_{dR})\in 2\pi i K+\BQ\log(\overline{\bold Q}^{\times})=
\BQ\log(\overline{\bold Q}^{\times}).
$$
Let $s_{dR}'$ be an arbitrary lifting. Then $s_{dR}'$ is written as
$s_{dR}'=s_{dR}+ab_{dR}$
with $a\in \BQ$. Thus we have
$$
(c(\Gamma),s_{dR}')=
(c(\Gamma),s_{dR})+
(c(\Gamma),ab_{dR})=
(c(\Gamma),s_{dR})+
a(c(\gamma),b_{dR})\in \BQ\log(\overline{\bold Q}^{\times})+\BQ.
$$
As for the precise argument, see \cite{AOT}.
\subsection{Period integrals for relative cohomologies 
and hypergeometric function}
We consider a sequence of morphisms:
\begin{equation}
\label{morphisms for coverings}
\begin{matrix}
X_{m} &\xrightarrow{\pi''} &\bold A^2
\\
(u,v,w)&\mapsto & (\xi,\eta)=(u^m,v^m) & 
\\
\cup & & \cup
\\
\Gamma''_1 & &\Gamma_1
\end{matrix}
\end{equation}

Let
$\Gamma_1$ be chains of $\bold A^2(\bold C)$ defined by
\begin{align*}
&\Gamma_1=\{(\xi,\eta)\in \bold R^2
\mid 0\leq \xi\leq 1, 0\leq \eta\leq 1,1\leq  \xi+\eta\},
\end{align*}
with the standard orientations and
$\Gamma''_1$ be a topological cycle 
on $X_{m}$ defined by
\begin{align*}
&\Gamma''_1=\bigg\{(u,v,w)\in X_{pm}\mid
\pi'\circ\pi(u,v,w)\in\Gamma_1, u,v,w\in \bold R_+\bigg\},
\end{align*}
Let $\gamma_0$ be a one chain defined by 
$$
\{t:[0,1]\to (1,t)\in X_{m}\}+
\{t:[0,1]\to (1-t,1)\in X_{m}\}
$$
For a chain $\beta$ in $X_m$, we set 
$$
pr_{\chi}(\beta)=\dfrac{1}{|G|}\sum_{g\in G}\chi(g)^{-1}g(\beta)
$$
\begin{proposition}
Under the above notation, we have
$pr_{\chi}(\Gamma_1)\in H_2^B(X,B,K)(\chi)$,
$pr_{\chi}(\gamma_0)\in H_1^B(B,K)(\chi)$ and
$$
\partial pr_{\chi}(\Gamma_1)=pr_{\chi}(\gamma_0).
$$
Moreover, $pr_{\chi}(\gamma_0)$ is a base of 
$H_1^B(B,K)(\chi)$.
\end{proposition}
We set
$$
\omega=(\xi+\eta-1)^{\alpha_1-1}\xi^{\alpha_2-1}\eta^{\alpha_3-1}.
$$
Since the pairing is given by integrals, we have
\begin{align*}
(c(pr_{\chi}(\Gamma_1)),\omega)=&\int_{pr_{\chi}(\Gamma_1)}\omega
=\dfrac{1}{|G|}\sum_g\chi^{-1}(g)\int_{g\Gamma_1}\omega
\\
=&\dfrac{1}{|G|}\sum_g\chi^{-1}(g)\int_{\Gamma_1}(g^{-1})^*\omega
=\int_{\Gamma_1}\omega
\end{align*}

In the following, we give an explicit formula for $(c(pr_{\chi}(\Gamma_1)),\omega)$
for the classification  (1)--(2) of 
Theorem \ref{Shioda's classification}.
The following proposition give a relation between the period integrals for relative cycles 
and special values of the hypergeometric function.
\begin{proposition}
\label{relative integral and hgf}
Let $\alpha_1, \alpha_2, \alpha_3>0$ be real numbers. Then we have
\begin{align*}
&\int_{\Gamma_1}(\xi+\eta-1)^{\alpha_1-1}\xi^{\alpha_2-1}\eta^{\alpha_3-1}
d\xi d\eta
\\
=&\dfrac{1}{(\alpha_1+\alpha_2)(\alpha_1+\alpha_3)}F(1,1,\alpha_1+\alpha_2+\alpha_3;\alpha_1+\alpha_2+1,\alpha_1+\alpha_3+1;1)
\end{align*}
\end{proposition}
\begin{proof}
By changing the variables
$$
\xi=\dfrac{1-s}{1-st}, \eta=\dfrac{1-t}{1-st},
$$
the domain $\Gamma_1$ of the integral are transformed into the domain
$$
\Gamma_1'=\{0<t_1<1,0<t_2<1\}.
$$
By changing the variable, we have
\begin{align*}
&\int_{\Gamma_1}(\xi+\eta-1)^{\alpha_1-1}\xi^{\alpha_2-1}\eta^{\alpha_3-1}
d\xi d\eta
\\
=&\int_{\Gamma'_1}(1-s)^{\alpha_1+\alpha_2-1}(1-t)^{\alpha_1+\alpha_3-1}(1-st)^{-\alpha_1-\alpha_2-\alpha_3}dsdt
\\
=&B(1,\alpha_1+\alpha_2)B(1,\alpha_1+\alpha_3)F(1,1,\alpha_1+\alpha_2+\alpha_3;\alpha_1+\alpha_2+1,\alpha_1+\alpha_3+1;1)
\end{align*}
\end{proof}
\subsection{First example}
We consider the case (1) in Theorem \ref{Shioda's classification},
$
(\alpha_0,\alpha_1,\alpha_2,\alpha_3)=(\alpha,1-\alpha, \beta,1-\beta),
$
where $\alpha,\beta \in \bold Q$ and $0<\alpha,\beta <1$.
We compute the integral
$$
\int_{\Gamma_1}
(\xi+\eta-1)^{-\alpha} \xi^{\alpha-1}\eta^{\beta-1}d\xi d\eta
$$
with
$\Gamma_1=\{\xi+\eta>1,\xi<1,\eta<1\}$.
By changing variable by
$\eta'=\dfrac{\xi+\eta-1}{\xi},\quad \xi'=\eta$
(i.e. $\xi=\dfrac{1-\xi'}{1-\eta'}, \eta=\xi'$), we have
\begin{align*}
\int_{\Gamma_1}
(\xi+\eta-1)^{-\alpha} \xi^{\alpha-1}\eta^{\beta-1}d\xi d\eta
=&\int_{\Gamma_1^*}\eta'^{-\alpha}\xi'^{\beta-1}
\dfrac{1}{1-\eta'}d\xi' d\eta'
\\
=&\int_{0}^1\eta'^{-\alpha}
\dfrac{1-\eta'^{\beta}}{\beta(1-\eta')}
 d\eta'
\end{align*}
where
$\Gamma_1^*=\{\eta'<\xi'<1,0<\eta'<1\}$.

We extend the above equality 
using analytic continuation using Pochhammer integral.
Let $\alpha\not\in \bold Z$ and $f(x)$ be a rational function of $\bold C$
without pole on $[0,1]\subset \bold R$. We define Pochhammer integral
by 
\begin{equation}
\label{Poch int}
\int_{P(0,1)}x^{\alpha}f(x)dx=
\int_{\epsilon}^1x^{\alpha}f(x)dx+
\dfrac{1}{\bold e(\alpha)-1}\int_{C_{\epsilon}}x^{\alpha}f(x)dx
\end{equation}
Here the path $C_{\epsilon}$ is defined by
$$
C_{\epsilon}:[0,1]\to \bold C:t\mapsto \epsilon\bold e(t).
$$
Here we choose the branch of $x^{\alpha}$ on $C_{\epsilon}$ to be
$\arg(x^\alpha)=2\pi \alpha t$ for the above parameter $t$.
Then the integral (\ref{Poch int}) is analytic function of $\alpha$
for $\alpha\not\in \bold Z$ and 
$$
\int_{0}^1x^{\alpha}f(x)dx=
\int_{P(0,1)}x^{\alpha}f(x)dx
$$
for $-1<\alpha$.
Thus we have the following theorem.

\begin{theorem}
Let $\alpha, \beta$ be real numbers such that $\alpha,\beta, \alpha-\beta\not\in \bold Z$.
We have the following identity:
\begin{align*}
\dfrac{1}{(-\alpha+1+\be)}
F(1, 1, 1+\be;2, -\alpha+2+\be;1)
=&\int_{P(0,1)}\eta^{-\alpha}
\dfrac{1-\eta^{\beta}}{\beta(1-\eta)}
 d\eta.
\end{align*}
\end{theorem}
By the following proposition, the above integral expression gives an explicit formula
as an element in 
$\overline{\bold Q}+\overline{\bold Q}\log(\overline{\bold Q}^{\times})$.
\begin{proposition}
Let $\alpha=\dfrac{n}{m}, \beta=\dfrac{n'}{m}\in (0,1)$ is a rational number.
We choose a suitable branch of $\log$
and a suitable path $[0,1]$.
Then we have
\begin{enumerate}
\item
Let 
$c\in \bold C^{\times}$, $|c|\neq 1$.
We choose $\gamma\in \bold C$ such that $\gamma^{m} =c$. then
\begin{align*}
\int_0^1\dfrac{x^\alpha}{c-x}\dfrac{dx}{x}=
-\sum_{i=0}^{m-1}\bold e(-ni/m)\log(1-\dfrac{\bold e(i/m)}{\gamma})
(=\dfrac{1}{c\alpha} \ _2F_1(1,\alpha;\alpha+1;1/c))
\end{align*}
\item
(Gauss's digamma theorem)
\begin{align*}
\int_0^1\dfrac{x^{\beta}-x^\alpha}{1-x}\dfrac{dx}{x}=
-\sum_{i=1}^{m-1}(\bold e(-n'i/m)-\bold e(-ni/m))\log(1-\bold e(i/m))
\end{align*}
\end{enumerate}
\end{proposition}

\begin{proof}
We have
\begin{align*}
\int_0^1\dfrac{x^\alpha}{c-x}\dfrac{dx}{x}=&
\int_0^1\dfrac{m\xi^{n-1}}{c-\xi^{m}}d\xi
=
\int_0^1\sum_{i=0}^{m-1}\dfrac{\bold e(-i(n-1)/m)}{\gamma-\bold e(i/m)\xi}d\xi
\\
=&
-\sum_{i=0}^{m-1}\bold e(-ni/m)\log(1-\dfrac{\bold e(i/m)}{\gamma})
\end{align*}
The second statement follows from the first.
\end{proof}

\subsection{Permutations of exponent indices}
Let $\Gamma_i$ ($i=1,2,3$) be chains of $\bold A^2(\bold C)$ defined by
\begin{align}
\label{def of Gamma}
&\Gamma_1=\{(\xi,\eta)\in \bold R^2
\mid 0\leq \xi\leq 1, 0\leq \eta\leq 1,1\leq  \xi+\eta\},
\\
\nonumber
&\Gamma_{2}=\{(\xi,\eta)\in \bold R^2
\mid 0\leq \xi\leq 1, -1 \leq \eta\leq 0, 
0 \leq \xi+\eta\leq 1\}
\\
\nonumber
&\Gamma_{3}=\{(\xi,\eta)\in \bold R^2
\mid -1\leq \xi\leq 0, 0 \leq \eta\leq 1, 
0 \leq \xi+\eta\leq 1\},
\end{align}
with the standard orientations.
We set $\xi=\eta', \eta=1-\xi'-\eta'$ and $\xi=1-\xi''-\eta'', \eta=\xi''$. 
Then we have

\begin{align*}
&\int_{\Gamma_{2}}(\xi+\eta-1)^{\alpha_1-1}\xi^{\alpha_2-1}\eta^{\alpha_3-1}
d\xi d\eta
\\
=
&(-1)^{\alpha_1+\alpha_3}\int_{\Gamma_1}(\xi')^{\alpha_1-1}(\eta')^{\alpha_2-1}(\xi'+\eta'-1)^{\alpha_3-1}
d\xi d\eta
\end{align*}
\begin{align*}
&\int_{\Gamma_{3}}(\xi+\eta-1)^{\alpha_1-1}\xi^{\alpha_2-1}\eta^{\alpha_3-1}
d\xi d\eta
\\
=
&(-1)^{\alpha_1+\alpha_2}\int_{\Gamma_1}(\eta'')^{\alpha_1-1}(\xi''+\eta''-1)^{\alpha_2-1}(\xi'')^{\alpha_3-1}
d\xi d\eta
\end{align*}
Thus the integral for $(\alpha_3,\alpha_1,\alpha_2)$ and 
$(\alpha_2,\alpha_3,\alpha_1)$ on $\Gamma_1\subset X_p$
is reduced to the integral for $(\alpha_1,\alpha_2,\alpha_3)$ on $\Gamma_{2}$ and $\Gamma_{3}$.

\begin{table}[H]
 \begin{tabular}{|c|c|c|c|}
\hline 
& $\Gamma_1$&   reduced to $\Gamma_{2}$ & reduced to $\Gamma_{3}$  
\\
\hline\hline 
$X_m$ &
$(\beta,-\alpha,\alpha)$  &$(\alpha,\beta,-\alpha)$ &$(-\alpha,\alpha,\beta)$  
\\   \hline 
$X_{2m}$ &
$(2\alpha,-\alpha,-\alpha+1/2)$  &$(-\alpha+1/2,2\alpha,-\alpha)$ &$(-\alpha,-\alpha+1/2,2\alpha)$ 
\\ \hline
$X_{3m}$ &
$(3\alpha,-\alpha+1/3,-\alpha+2/3)$  & $(-\alpha+2/3,3\alpha,-\alpha+1/3)$ &$(-\alpha+1/3,-\alpha+2/3,3\alpha)$ 
\\\hline
$X_{4m}$ &
$(4\alpha,-\alpha+1/4,-\alpha+3/4)$  & $(-\alpha+3/4,4\alpha,-\alpha+1/4)$ &$(-\alpha+1/4,-\alpha+3/4,4\alpha)$ 
\\\hline
\end{tabular}
\end{table}

We have a symmetry on 
$(\alpha_1,\alpha_2,\alpha_3)$ and
$(\alpha_1,\alpha_3,\alpha_2)$.
Therefore
$(X_{2m},\Gamma_{3})$-case is reduced to 
$(X_{2m},\Gamma_{2})$-case by the shifting $\alpha\mapsto\alpha-1/2$.
Similarly, $(X_m,\Gamma_{2})$-case is reduced to 
$(X_m,\Gamma_{3})$-case.
The cases $(X_m,\Gamma_{3})$ is computed in the previous example.
As for the relation with Watson's formula, see \cite{AOT}.
We set $\alpha_0=(\alpha_1\alpha_2\alpha_3)^{-1}$.
Let $(\alpha_{i_0},\alpha_{i_1},\alpha_{i_2},\alpha_{i_3})$
be a quadruple obtained by a permutation of 
$(\alpha_{0},\alpha_{1},\alpha_{2},\alpha_{3})$.
Then $\alpha'=(\alpha_{i_1},\alpha_{i_2},\alpha_{i_3})$-part is also generated by algebraic cycles if 
$\alpha=(\alpha_{1},\alpha_{2},\alpha_{3})$-part is.
The corresponding period integral is an easy linear combination of
periods corresponding to permutations of $(\alpha_{1},\alpha_{2},\alpha_{3})$.

\section{Algebraic cycles and one forms with constant residues}
\subsection{Algebraic cycles on Fermat surfaces}
As for the case (1)--(2) in Theorem \ref{Shioda's classification}, 
algebraic cycles generating the Hodge classes
are given as follows (see \cite{AS}).
We define varieties $Z^{(1)}_{i,j}$,
$Z^{(2)}_{i\pm}$,
$Z^{(3)}_{i,j,k}$ and $Z^{(4)}_{i,k,\pm}$ by
\begin{align*}
&Z^{(1)}_{i,j}:w=\bold e(i/m+1/2m),\quad  u=\bold e(j/m+1/2m)v,\quad
(0\leq i,j \leq m-1)
\\
&Z^{(2)}_{i\pm}:w^2=(\mp 1)^{1/m}\bold e(i/m)2^{1/m}uv,\quad u^m\pm v^m=1,
\quad
(0\leq i\leq m-1)
\\
&Z^{(3)}_{i,j,k}:w^{3}=3^{1/m} \bold e((i+j)/3m+k/m+1/2m) u v,
\quad \omega^i u^m +\omega^{j} v^m=1
\\
&\qquad  (0\leq i,j \leq 2, 0\leq k\leq m-1)
\\
&Z^{(4)}_{i,k,\pm}:w^4=\bold e((k+4i)/4m)(2\sqrt{2})^{1/m}uv,\quad
\\
&\qquad u^{2m}+\bold e(k/2)v^{2m}\mp \bold e(k/4)\sqrt{2}u^mv^m=\pm 1,
\quad (0\leq i \leq m, 0\leq k\leq 3)
\end{align*}
\begin{proposition}
We have the following inclusions
$$
Z^{(1)}_{i,j}\subset X_m,\quad Z^{(2)}_{\pm,i}\subset X_{2m}, \quad
Z^{(3)}_{i,j,k}\subset X_{3m},\quad
Z^{(4)}_{i,k,\pm}\subset X_{4m}
$$
\end{proposition}
\begin{proof}
We set $x=u^m, y=v^m, z=w^m$. Using relations
\begin{align*}
z&=-1,\quad y+y=0\qquad(\text{ on }Z^{(1)}_{i,j})
\\
z^2&=\mp 2x y,\quad x \pm y=1\qquad(\text{ on }Z^{(2)}_{\pm,i})
\\
z^3&=-3 \omega^{i+j} x y,
\quad \omega^i x +\omega^{j} y=1
\qquad(\text{ on }Z^{(3)}_{i,j,k})
\\
z^4&=2\sqrt{2} i^{k} x y,
\quad x +(-1)^{k} y \mp i^k xy \mp 1=0
\qquad(\text{ on }Z^{(4)}_{i,k,\pm})
\end{align*}
we have
\begin{align*}
&x^2+y^2-z=1
\quad(\text{ on }Z^{(1)}_{i,j})
\\
&x^2+y^2-z^2-1=x^2+ y^2 \pm 2xy -1
=(x\pm y-1)(x\pm y+1)
\quad(\text{ on }Z^{(2)}_{\pm,i})
\\
& x^3+y^3-z^3-1=
x^3+y^3+3\omega^{i+j} xy-1
\\
&=
(\omega^ix+\omega^j y-1)
(\omega^{i+1}x+\omega^{j-1} y-1)
(\omega^{i-1}x+\omega^{j+1} y-1)=0,
\quad(\text{ on }Z^{(3)}_{i,j,k})
\\
&x^4+y^4-z^4-1= 
x^4+y^4-i^k2\sqrt{2}xy-1
 \\
&=
(x^2+(-1)^ky^2-i^k \sqrt{2}xy-1) 
(x^2+(-1)^ky^2+i^k \sqrt{2}xy+1)=0,
\quad (\text{ on }Z^{(4)}_{i,k,\pm})
\end{align*}
Thus we have the required inclusions.
\end{proof}
By the above proposition, we have
\begin{align*}
&Z_1=\bigcup_{i,j} Z^{(1)}_{i,j}\subset X_m,\quad
Z_2=\bigcup_{i} Z^{(2)}_{i,\pm}\subset X_{2m},
Z_3=\bigcup_{i,j,k} Z^{(3)}_{i,j,k}\subset X_{3m},\quad
Z_4=\bigcup_{i,k} Z^{(4)}_{i,k,\pm}\subset X_{4m}
\end{align*}

Then the varieties $Z$ is stable under the action 
of $G_{pm}=\mu_{pm}^3$
given as (\ref{Fermat abel action}).
We set $x=u^m, y=v^m, z=w^m$ and for an element of $\alpha \in \dfrac{1}{pm}\bold Z$
and use a notation $x^{\alpha}=u^{m\alpha}$, etc.
We define characters $\phi_p$ of $G_{pm}$ by setting $\phi_p=\chi_{\bold a}$, where
$$
\bold a=\begin{cases}
(\alpha,-\alpha,\beta,-\beta)
&\quad(p=1)
\\
(2\alpha,1-\alpha,-\alpha+\dfrac{1}{2},\dfrac{1}{2})
&\quad(p=2)
\\
(3\alpha,1-\alpha,-\alpha+\dfrac{1}{3},-\alpha+\dfrac{2}{3})
&\quad(p=3)
\\
(4\alpha,1-2\alpha,-\alpha+\dfrac{1}{4},-\alpha+\dfrac{3}{4})
&\quad(p=4)
\end{cases}
$$
Here the Kummer character $\chi_{\bold a}$ is defined by
(\ref{Kummer char on Fermat})
The following proposition is proved in \cite{AS}.
\begin{proposition}[\cite{AS}]
\label{Aoki}
The $\phi_p$ parts of the cohomology classes of the components in 
$Z_p$ generates $H^2(X_{pm})(\phi_p)$.
In other words, the map
$H^2_{Z_p}(X_{pm})(\phi_p)
 \to 
H^2(X_{pm})(\phi_p)
$ is surjective.
As a consequence, the map
$$
H^2(X_{pm})(\phi_p)
 \to 
H^2(X_{pm}-Z_p)(\phi_p)
$$
is the zero map
by the localization exact sequence.
\end{proposition}

\subsection{Algebraic cycles and differential forms}
Let $p=1,2,3$. 
We express the de Rham cohomology $H^2(X_{pm})(\phi_p)$
by that of an open set $X^0_p$ defined by
$$
X_{pm}^0=\{(u,v,w)\in X_{pm}\mid uvw\neq 0\}
$$
By branch condition, we have
$$
H^2(X_{pm})(\phi_p)
\simeq H^2(X_{pm}^0)(\phi_p),\quad
H^2(X_{pm}-Z_p)(\phi_p)
\simeq H^2(X_{pm}^0-Z_p)(\phi_p).
$$
Let $\Omega=\Omega_p \in \Gamma(X_{pm}^0,\Omega^2)(\phi_p)$ 
be a closed form whose cohomology class 
generates 
$H^2(X_{pm})(\phi_p)$.
We set $U_p=X_{pm}^0-Z_p$ and the inclusion $U_p \to X_{pm}^0$ is denoted by $j$.
We consider the following diagram:
$$
\begin{matrix}
&\Gamma(X_{pm}^0,\Omega^1_{X_{pm}})(\phi_p) &\xrightarrow{d}&
\Gamma(X_{pm}^0,\Omega_{X_{pm}}^2)(\phi_p)
\\
&
\phantom{j^*}\downarrow j^*
& &
\phantom{j^*}\downarrow j^*
\\
&\Gamma(X_{pm}^0,\Omega_{X_{pm}}^1(\log Z_p))(\phi_p) &\xrightarrow{d}&
\Gamma(X_{pm}^0,\Omega_{X_{pm}}^2(\log Z_p))(\phi_p)
\\
&\operatorname{res}\downarrow\phantom{\operatorname{res}} 
& & \phantom{\operatorname{res}}\downarrow\operatorname{res}
\\
&
\oplus_{Z_p^{(i)}\subset Z_p}\Gamma(X^0_p\cap Z^{(i)}_p,\Cal O)(\phi_p) &\xrightarrow{d}&
\oplus_{Z_p^{(i)}\subset Z_p}\Gamma(X^0_p\cap Z^{(i)}_p,\Omega^1)(\phi_p)
\end{matrix}
$$

By Proposition \ref{Aoki}, there exists an element 
$\psi\in \Gamma(X_{pm}^0,\Omega_{X_{pm}}^1(\log(Z_p))(\phi_p)$ such that
$$
d\psi=j^*\Omega.
$$
By the above diagram, $res(\psi)$ is a locally constant function.
Vice versa, if $res(\psi)$ of a one form $\psi$ is a constant function 
on each component, then $\Omega=d\psi$ is a closed form and
an image under the map
$$
j^*:\maketitle
\Gamma(X_{pm}^0,\Omega_{X_{pm}}^2)(\phi_p)
\to \Gamma(X_{pm}^0,\Omega_{X_{pm}}^2(\log Z_p))(\phi_p)
$$
Moreover, if the class
$$
[\Omega]\in H^2(X_{pm}^{0})(\phi_p)\simeq H^2_{Z_P}(X_{pm}^{0})(\phi_p)
$$ 
is not zero,
it becomes a base.
In the next subsection, we construct $\psi$ satisfying the above equation.



\subsection{One forms $\psi$ with constant residues along $Z_{p}$}

For $p=1,2,3,4$, we consider varieties 
$$
X_{p}=
Spec(\bold C[x,y]/(x^p+y^p-1-z^p))
$$
and consider a sequence of morphisms:
\begin{equation}
\label{three varieties}
\begin{matrix}
X_{pm} &\xrightarrow{\pi} &X_{p} &\xrightarrow{\pi'} &\bold A^2
\\
(u,v,w)&\mapsto & (x,y,z)=(u^m,v^m,w^m) & 
\\
& & (x,y,z)&\mapsto & (\xi,\eta)=(x^p,y^p)
\end{matrix}
\end{equation}
We define one form $\psi_0$ on $X_{pm}$ by
$$
\psi_0=
\begin{cases}
(x+y-1)^\beta\dfrac{dx+dy}{x+y},
&\quad\text{ on }X_m
\\
\dfrac{dx+dy}{x+y-1}-\dfrac{dx-dy}{x-y-1},
&\quad\text{ on }X_{2m}
\\
\sum_{1\leq i,j \leq 3}\omega^{2i+j}\dfrac{\omega^idx+\omega^jdy}
{\omega^ix+\omega^jy-1},
&\quad\text{ on }X_{3m}
\\
\sum_{0\leq k \leq 3}
i^kd\log\big(x^2+(i^{k}y)^2+\sqrt{2}x(i^{k}y)+1\big)
\\
\quad-\sum_{0\leq k \leq 3}
i^kd\log\big(x^2+(i^{k}y)^2-\sqrt{2}x(i^{k}y)-1\big),
&\quad\text{ on }X_{4m}
\end{cases}
$$
and rational functions $f$ by
$$
f=\begin{cases}
\dfrac{y}{x},
&\quad\text{ on }X_m
\\
\dfrac{(x^2+y^2-1)^2}{x^2y^2},
&\quad\text{ on }X_{2m}
\\
\dfrac{(x^3+y^3-1)^3}{x^3y^3},
&\quad\text{ on }X_{3m}
\\
\dfrac{(x^4+y^4-1)^4}{x^4y^4}.
&\quad\text{ on }X_{4m}
\end{cases}
$$
Since
$f|_{Z_p}$ is a constant function,
the residue of a differential form 
$\psi=f^\alpha\psi_0$
is constant on each component of $Z_p$.
Using relations
$d\psi_0=0$ and
$$
df\wedge \psi_0=
\begin{cases}
-(x+y-1)^\beta\dfrac{dx\wedge dy}{x^2},
&\quad\text{ on }X_m
\\
-4\dfrac{(x^2+y^2-1)dx\wedge dy}{x^3y^2},
&\quad\text{ on }X_{2m}
\\
-\dfrac{27(x^3+y^3-1)^2dx\wedge dy}{x^3y^2},
&\quad\text{ on }X_{3m}
\\
-\dfrac{64\sqrt{2}(x^4+y^4-1)^3dx\wedge dy}{x^4y^2},
&\quad\text{ on }X_{4m}
\end{cases}
$$
we have
$$
d\psi=\alpha f^{\alpha-1}df\wedge\psi_0
=\begin{cases}
-\alpha
y^{\alpha-1}
x^{-\alpha-1}
(x+y-1)^{\beta} 
dx\wedge dy,
&\quad\text{ on }X_m
\\
-4\alpha
(x^2+y^2-1)^{2\alpha-1} 
x^{-2\alpha-1}
y^{-2\alpha}
dx\wedge dy,
&\quad\text{ on }X_{2m}
\\
-27\alpha
(x^3+y^3-1)^{3\alpha-1} 
x^{-3\alpha}
y^{-3\alpha+1}
dx\wedge dy,
&\quad\text{ on }X_{3m}
\\
-64\sqrt{2}\alpha
(x^4+y^4-1)^{4\alpha-1} 
x^{-4\alpha}
y^{-4\alpha+2}
dx\wedge dy,
&\quad\text{ on }X_{4m}.
\end{cases}
$$
By setting $x^p=\xi, y^p=\eta$, we have
the following multi valued differential form on $\bold A^2$:
$$
\Omega=d\psi=
\begin{cases} 
-\alpha
\eta^{\alpha-1}
\xi^{-\alpha-1}
(\xi+\eta-1)^{\beta} 
d\xi\wedge d\eta,
&\quad\text{ on }X_m
\\
-\alpha(\xi+\eta-1)^{2\alpha-1} 
\xi^{-\alpha-1}
\eta^{-\alpha-1/2}
d\xi\wedge d\eta,
&\quad\text{ on }X_{2m}
\\
-3\alpha
(\xi+\eta-1)^{3\alpha-1} 
\xi^{-\alpha-2/3}
\eta^{-\alpha-1/3}
d\xi\wedge d\eta,
&\quad\text{ on }X_{3m}
\\
-4\sqrt{2}\alpha
(\xi+\eta-1)^{4\alpha-1} 
\xi^{-\alpha-3/4}
\eta^{-\alpha-1/4}
d\xi\wedge d\eta,
&\quad\text{ on }X_{4m}.
\end{cases}
$$
Then the integral $\displaystyle\int_{\Gamma_i}\Omega$
converges, if $\alpha$ and $\beta$ are real numbers and
satisfy
the following conditions.
\begin{table}[H]
 \begin{tabular}{|c|c|c|c|c|}
\hline 
& $\Gamma_1$ & $\Gamma_{2}$ & $\Gamma_{3}$ 
\\
\hline\hline 
$X_m$ 
&$-1+|\alpha|<\beta$ &  &
\\   \hline 
$X_{2m}$ 
&$0<\alpha$  &$0<\alpha<1/4$ & 
\\ \hline
$X_{3m}$ 
&$0<\alpha$   & $-1/3<\alpha<1/2$ &$-1/6<\alpha<1/3$  
\\\hline
$X_{4m}$ 
&$0<\alpha$  &$-1/4<\alpha<1/2$ &$-1/12<\alpha<1/4$ 
 \\\hline
\end{tabular}
\caption{Convergent condtion}
\label{conv cond}
\end{table}

\section{Hypergeometric identities for $\Gamma_1$.}

\subsection{Blowing up and Stokes' formula}

For a real number $\alpha$, we set 
$(-1)^{\alpha}=\bold e(\alpha/2)$. Let
$\Gamma_i$ and be ($i=1,2,3$) topological chains defined in 
(\ref{def of Gamma}).
We define a topological cycle 
$\Gamma''_1, \Gamma''_{2},\Gamma''_{3}$ on $X_{pm}$
for $p=1,2,3,4$ by
\begin{align}
\label{def of Gamma prpr}
&\Gamma''_1=\bigg\{(u,v,w)\in X_{pm}\mid
\pi'\circ\pi(u,v,w)\in\Gamma_1, u,v,w\in \bold R_+\bigg\},
\\
\nonumber
&\Gamma''_{2}=\bigg\{(u,v,w)\in X_{pm}\mid
\pi'\circ\pi(u,v,w)\in \Gamma_{2}, 
w,v \in (-1)^{1/pm}\bold R_+,
u\in \bold R_+
\bigg\},
\\
\nonumber
&\Gamma''_{3}=\bigg\{(u,v,w)\in X_{pm}\mid
\pi'\circ\pi(u,v,w)\in \Gamma_{3}, 
w,u \in (-1)^{1/pm}\bold R_+,
v\in \bold R_+
 \bigg\}.
\end{align}
Here $\pi'\circ\pi$ is a morphism defined in (\ref{three varieties}).
The image $\pi(\Gamma''_i)$ of $\Gamma''_i$ under the map $\pi$ 
defined in (\ref{three varieties})
is denoted by $\Gamma_i'$. 
Then we have the following diagram:
$$
\begin{matrix}
X_{pm} &\xrightarrow{\pi} &X_{p} &\xrightarrow{\pi'} &\bold A^2
\\
\cup & & \cup& & \cup
\\
\Gamma_i'' & & \Gamma_i' & & \Gamma_i
\end{matrix}
$$
We set $B_y=\{x=(-1)^{1/p}y\}, B_x=\{y=(-1)^{1/p}x\}, B_{2}=\{x=1\}$ and $B_{3}=\{y=1\}$.
Then we have
\begin{align*}
&\partial \Gamma'_1 \subset \{x^p+y^p=1\}\cup B_{2}\cup B_{3},
\\
&\partial \Gamma'_{2} \subset B_x\cup B_{2}\cup \{y=0\}.
\\
&\partial \Gamma'_{3} \subset B_y\cup \{x=0\}\cup B_{3},
\end{align*}

In the last section, we have a relation $d\psi=j^*\Omega$ on $X_{p}$.
Since the pole of the differential form $\psi_p$ passes through the singular
points 
$$
p_{2}=\{(x,y)=(1,0)\},\quad
p_{3}=\{(x,y)=(0,1)\}
$$
in $\Gamma'_{2}, \Gamma'_{3}$,
we can not apply Stokes' formula for $\Gamma'_i$ in $X_p$. Thus we consider the blowing up 
$b:\widetilde{X_p} \to X_p$
of the variety $X_p$
with the center $B'\cap Z'$. Let $E_{2}$ and $E_{3}$ be the exceptional divisors over the points $p_{2}$ and
$p_{3}$, respectively.

\vskip 0.3in
\hskip 0.3in\includegraphics[width=5cm]{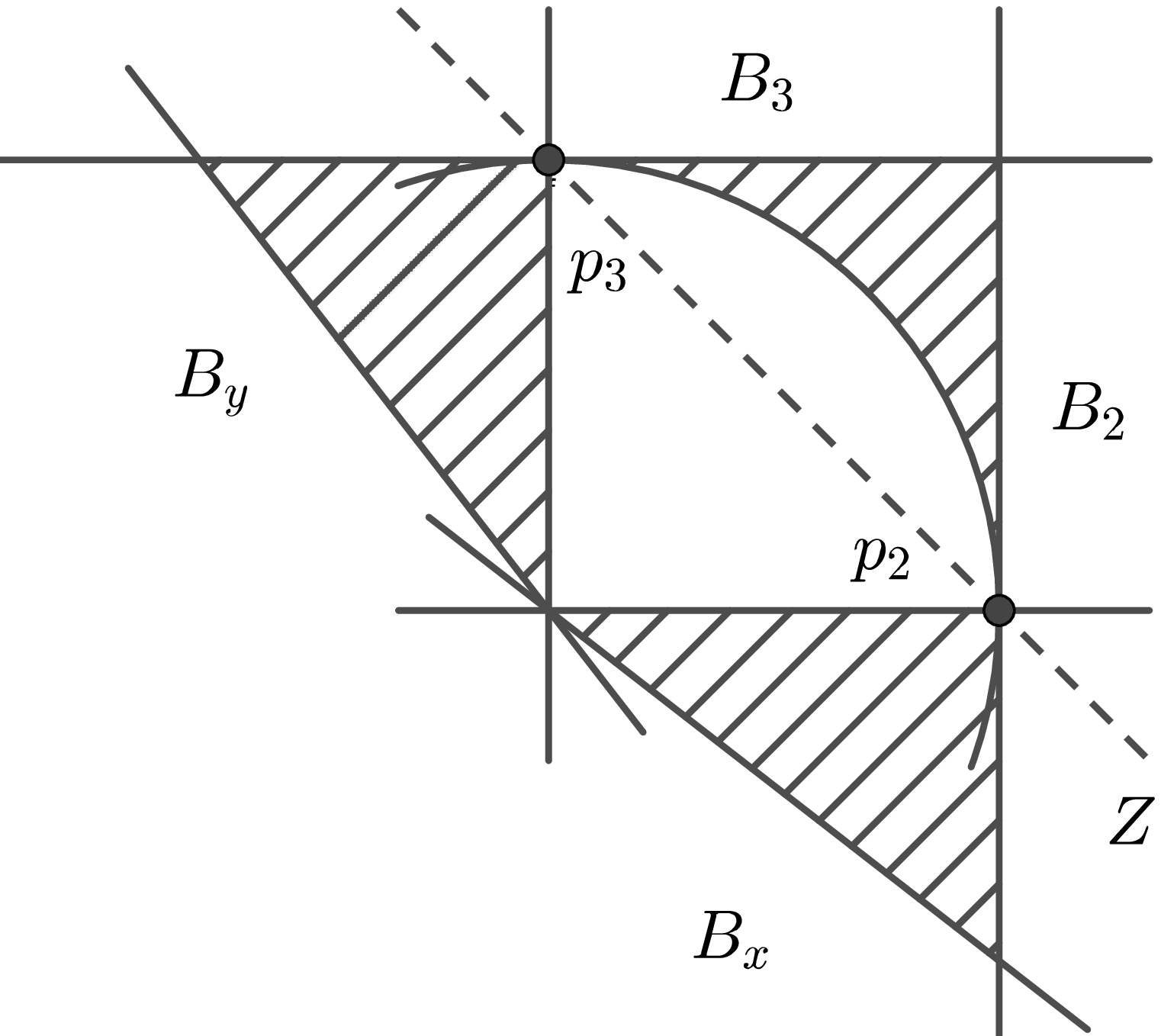}
\hskip 0.4in\includegraphics[width=4.8cm]{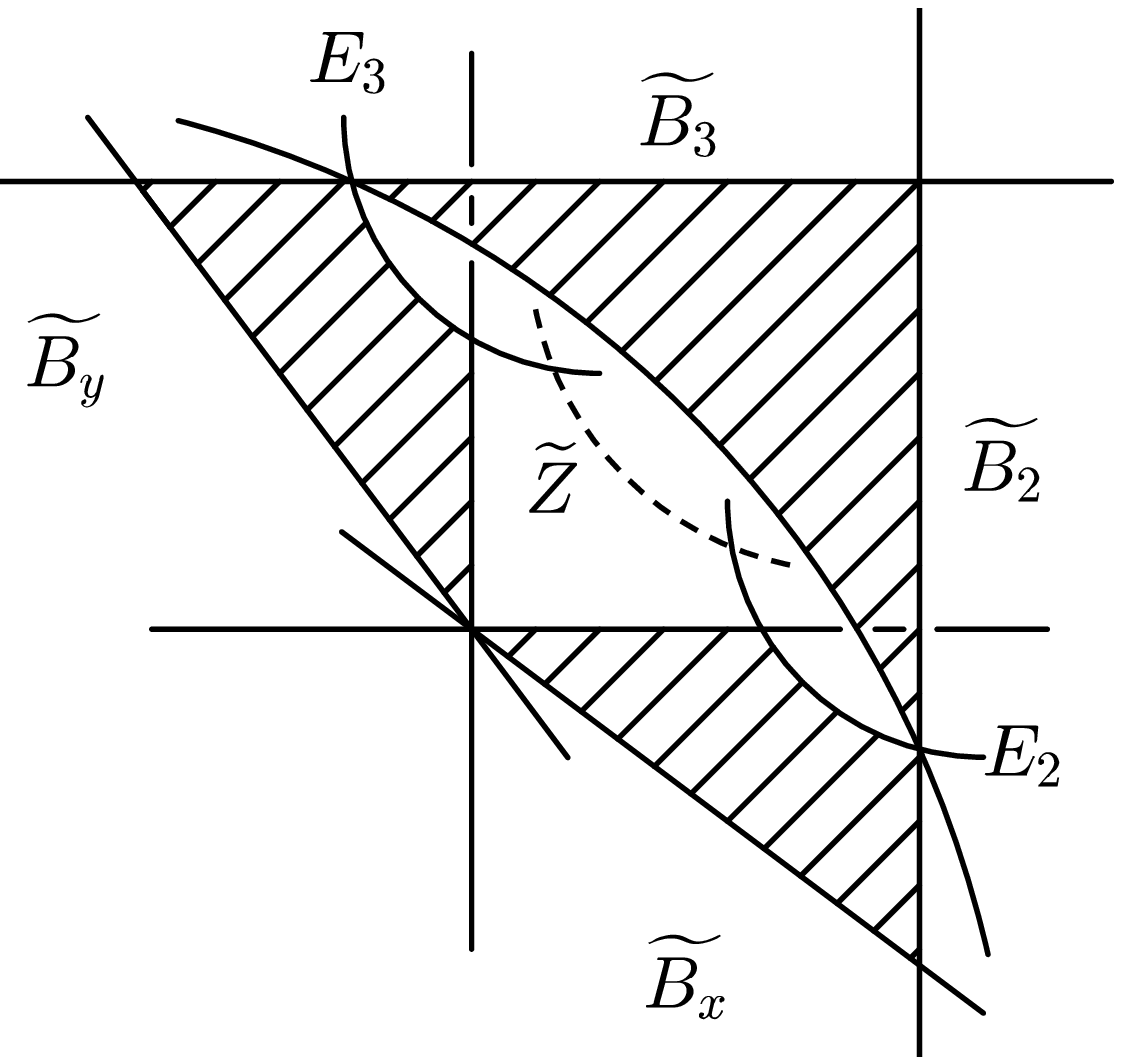}

\vskip 0.2in

Let $\widetilde{\Gamma_i}$ be the closure of $b^{-1}({\Gamma_i'}^0)$ where
${\Gamma_i'}^0$ is the relative interior of $\Gamma_i'$. The chain $\widetilde{\Gamma_i}$ is called 
the proper transform of
$\Gamma_i'$ for short.
From now on we assume that $\alpha$ and $\beta$ satisfy the convergent condition
Table \ref{conv cond}.
Via the birational map $b$, we have an equality 
$$
\int_{\Gamma_i'} \Omega =  \int_{\widetilde{\Gamma_i}} \widetilde{\Omega}.
$$
Let $\widetilde{B_i}$ and $\widetilde{R}$ be the proper transform of $B_i$ ($i=2,3,x,y$)
and $R=\{x^p+y^p=1\}\cup \{x=0\}\cup \{y=0\}$, respectively.
Then we have
\begin{align*}
&\partial\widetilde{\Gamma_1}\subset \widetilde{B_{3}}\cup \widetilde{B_{2}} \cup \widetilde{R},
\\
&\partial\widetilde{\Gamma_{2}}\subset \widetilde{B_{x}}\cup \widetilde{B_{2}} \cup \widetilde{R} \cup E_{2},
\\
&\partial\widetilde{\Gamma_{3}}\subset \widetilde{B_{y}}\cup \widetilde{B_{3}} \cup \widetilde{R} \cup E_{3}.
\end{align*}
Since $d\widetilde{\psi_p}=\widetilde{\Omega}$ on $\widetilde{X_{p}}$ 
and the pole of the rational differential form $\widetilde{\psi_p}$ 
does not intersect with the proper transforms 
$\widetilde{\Gamma_i}$ for $i=2,3$, we have
\begin{equation}
\label{Stokes on blow up}
\int_{\widetilde{\Gamma_i}} \widetilde{\Omega}=\int_{\partial\widetilde{\Gamma_i}}\widetilde{\psi_p}
\end{equation}
by Stokes' formula.

\subsection{Computation of restrictions for $\Gamma_1$}
The restriction of the rational one form $\psi$ to $\widetilde{B_i}$ $(i=x,y,2,3)$ and $E_j$ $(j=1,2)$ are denoted by 
$\psi_{B_i}=\psi|_{B_i}$ and $\psi_{E_j}=\psi|_{E_j}$, respeictively.
We have 
$\psi\mid_{x^p+y^p=1}=0$ on $X_{p}$ for $p=1, 2, 3$.
The restrictions $\psi_{B_{2}}$, $\psi_{B_{3}}$ are computed as follows.
\begin{align*}
&\text{ on }X_m:
\psi_{B_{2}}=
y^{\beta+\alpha} 
\dfrac{dy}{1+y},\quad
\psi_{B_{3}}=
x^{\beta-\alpha} 
\dfrac{dx}{1+x},
\quad 
\\
&\text{ on }X_{2m}:
\psi_{B_{2}}=
0,\quad
\psi_{B_{3}}=
2x^{2\alpha-1} 
\dfrac{dx}{2-x},
\quad 
\\
&\text{ on }X_{3m}:
\psi_{B_{2}}=
-y^{6\alpha}
\dfrac{9y(3+y^3)dy}{27+y^6},\quad
\psi_{B_{3}}=
x^{6\alpha}
\dfrac{9(9+x^3)dx}{27+x^6},
\quad 
\\
&\text{ on }X_{4m}:\begin{cases}&\psi_{B_{2}}=
-y^{12\alpha}\dfrac{8 \sqrt{2} y^2 (-8+4 y^4+y^8)}{-64+y^{12}} dy,
\\
&\psi_{B_{3}}=
x^{12\alpha}\dfrac{8 \sqrt{2} (-8 x^4-32+x^8)}{-64+x^{12}} dx.
\quad 
\end{cases}
\end{align*}

\subsection{Stokes' theorem for $\Gamma_1$}
By Stokes' theorem \ref{Stokes on blow up}, and the computations for the restriction of $\eta$, we have
the following theorem.
\begin{theorem}
\label{stokes consequence Gamma1}
Suppose that $\alpha, \beta$ be real number satisfying the condition
of Table \ref{conv cond}. Then we have the following equalities.
\begin{enumerate}
\item
For $-1+|\alpha|<\beta$, we have
$$
\alpha\int_{\Gamma_1}
(x+y-1)^{\beta}
x^{-\alpha-1}
y^{\alpha-1} 
dx\wedge dy
=-\int_0^1y^{\beta+\alpha} 
\dfrac{dy}{1+y}+
\int_0^1x^{\beta-\alpha} 
\dfrac{dx}{1+x}
$$
\item
For $\alpha>0$, we have
$$
\alpha\int_{\Gamma_1}(\xi+\eta-1)^{2\alpha-1} 
\xi^{-\alpha-1}
\eta^{-\alpha-1/2}
d\xi\wedge d\eta
=
2\int_0^1x^{2\alpha-1} 
\dfrac{dx}{2-x}
$$
\item
For $\alpha>0$, we have
\begin{align*}
&\alpha\int_{\Gamma_1}
(\xi+\eta-1)^{3\alpha-1} 
\xi^{-\alpha-2/3}
\eta^{-\alpha-1/3}
d\xi\wedge d\eta
\\
=&3\int_0^1y^{6\alpha}
\dfrac{y(3+y^3)dy}{27+y^6}
+3\int_0^1x^{6\alpha}
\dfrac{(9+x^3)dx}{27+x^6}
\end{align*}
\item
For $\alpha>0$, we have
\begin{align*}
&\alpha
\int_{\Gamma_1}(\xi+\eta-1)^{4\alpha-1} 
\xi^{-\alpha-3/4}
\eta^{-\alpha-1/4}
d\xi\wedge d\eta
\\
=&
2\int_0^1y^{12\alpha}\dfrac{y^2 (-8+4 y^4+y^8)}{64-y^{12}} dy 
+2\int_0^1x^{12\alpha}\dfrac{-8 x^4-32+x^8}{64-x^{12}} dx 
\end{align*}
\end{enumerate}
\end{theorem}

By anlytinc continuation, we have the following hypergeometric identities.
\begin{theorem}
\begin{enumerate}
\item
For $\alpha \not \in \bold Z$, we have the following equalities.
\begin{align*}
&
\dfrac{\alpha}{(\beta-\alpha+1)(\be+1+\alpha)}
F(1, 1, \be+1;\be-\alpha+2, \be+2+\alpha;1)
\\
=&
\int_{P(0,1)} 
\dfrac{x^{\beta-\alpha}-x^{\beta+\alpha}}{1+x}dx
\end{align*}
\item
For $2\alpha \not \in \bold Z$, we have the following equalities.
\begin{align*}
\dfrac{1}{\alpha+1/2}F(1, 1, 1/2;\alpha+1, \alpha+3/2;1)
=&
2\int_{P(0,1)} 
\dfrac{x^{2\alpha-1}dx}{2-x}
\end{align*}
\item
For $3\alpha \not \in \bold Z$, we have the following equalities.
\begin{align*}
&\dfrac{\alpha}{(6\alpha+1)(2 \alpha+2/3)}
F(1, 1, \alpha+1;2 \alpha+4/3, 2 \alpha+5/3;1)
\\
=&
\int_{P(0,1)}y^{6\alpha+1}
\dfrac{(3+y^3)dy}{27+y^6}
+\int_{P(0,1)}x^{6\alpha}
\dfrac{(9+x^3)dx}{27+x^6}
\end{align*}
\item
For $4\alpha \not \in \bold Z$, we have the following equalities.
\begin{align*}
&\dfrac{4 \alpha}{(12\alpha+1)(3 \alpha+3/4)}
F(1, 1, 2 \alpha+1;3 \alpha+5/4, 3 \alpha+7/4;1)
\\
=&
\int_{P(0,1)}y^{12\alpha}\dfrac{2 y^2 (8-4 y^4-y^8)}{64-y^{12}} dy 
+\int_{P(0,1)}x^{12\alpha}\dfrac{2 (8 x^4+32-x^8)}{64-x^{12}} dx 
\end{align*}
\end{enumerate}
\end{theorem}
\section{Hypergeometric identities for $\Gamma_{2}$}

\subsection{Restriction of $\psi$ to the boundary for $\Gamma_{2}$}
In this section, we give a hypergeometric identities arising from
the integration on $\Gamma_{2}$

\subsubsection{Restriction to the component $B_x$}
On $X_{pm}$, we have
\begin{align*}
&\psi_{B_x}=
\begin{cases}
(-1)^{\alpha}\dfrac{-2ix^{-4\alpha}dx}{2x^2-2x+1},
&\quad p=2
\\
-(-1)^{2\alpha}
\dfrac{27\omega x^{-6\alpha+2}(3x^3+1)}{1+27x^6}
dx,
&\quad p=3
\\
(-1)^{3\alpha}x^{-8\alpha}
\dfrac{(32-32i)(1+8x^4)x^3}{64x^8+1}dx,
&\quad p=4
\end{cases}
\end{align*}

\subsubsection{Restriction to the component $B_{2}$}
We restrict $\psi$ to $B_{2}$ and change parameter of integrals on $X_{pm}$
by $y=(-1)^{1/p}x$. Then we have
\begin{align*}
 \psi_{B_{2}}
=&\begin{cases}
0,
&\quad p=2
\\
-(-1)^{2\alpha}{x}^{6\alpha}
\dfrac{9\omega {x}(3-{x}^3)d{x}}{27+{x}^6},
&\quad p=3
\\
-(-1)^{3\alpha}x^{12\alpha} 
\dfrac{(8-8 i)x^2 (-8-4 x^4+x^8)}{64+x^{12}} dx, 
&\quad p=4
\end{cases}
\end{align*}
\subsubsection{Contribution from an exceptional divisor on $\Gamma_{2}$}

By setting $x-1=uy$, we compute the contribution of $\psi$ from the 
exceptional divisor $E_{2}$ on $X_{2m}$ at $x-1=y=0$ as follows:
\begin{align*}
\psi=
&2\dfrac{(x^2+y^2-1)^{2\alpha}(-ydx+(x-1)dy)}
{(x^2y^2)^{\alpha}((x-1)^2-y^2)}
\\
=&2\dfrac{(u^2y^2+2uy+y^2)^{2\alpha}(-yudy-y^2du+uydy)}
{((uy+1)^2y^2)^{\alpha}(u^2-1)y^2},
\\
\psi_{E_{2}}=&
2\dfrac{(2u)^{2\alpha}du}
{1-u^2}
=2\dfrac{(4u^2)^{\alpha}du}
{1-u^2}.
\end{align*}
Changing the coordinates by $u=-\bold e(-1/2p)v$,
($v \in \bold R_+$), we have
$$
\psi_{E_{2}}=
(-1)^{\alpha}\dfrac{2i(4v^2)^{\alpha}dv}{v^2+1}.
$$
Similarly on $X_{3m}$ and $X_{4m}$,
using parameters $x-1=uy$, $u=-\bold e(-1/2p)v$ ($v \in \bold R_+$), 
we have
$$
\psi_{E_{2}}=
\begin{cases}
(-1)^{2\alpha}\omega\dfrac{3(27v^3)^{\alpha}dv}{v^3+1},
&\text{ on }X_{3m}
\\
-(-1)^{3\alpha}(4v)^{4\alpha}\dfrac{(4-4 i) }{4 v^4+1}dv,\quad
&\text{ on }X_{4m}.
\end{cases}
$$
Therefore, we have
$$
\int_0^{\infty}\psi_{E_{2}}
=\begin{cases}
(-1)^{\alpha}4^{\alpha}\dfrac{\pi i}
{\cos(\pi\alpha)},
&\text{ on }X_{2m}
\\
(-1)^{2\alpha}\omega\dfrac{3^{3\alpha}\pi}{\sin(\pi\alpha+\pi/3)},
&\text{ on }X_{3m}
\\
-(-1)^{3\alpha}4^{3\alpha}\dfrac{(1/2-1/2 i)  
\sqrt{2} \pi}{\sin(\pi \alpha+\pi/4)},
&\text{ on }X_{3m}
\end{cases}
$$

\subsection{Hypergeometric identities for $\Gamma_{2}$} 
\subsubsection{Hypergeometric identities for $X_{2m}$} 
By Stokes's theorem on $X_{2m}$, we have the following 
equality:
\begin{align*}
&\alpha\int_{\Gamma_{2}}(\xi+\eta-1)^{2\alpha-1} 
\xi^{-\alpha-1}
\eta^{-\alpha-1/2}
d\xi\wedge d\eta
\\
=&(-1)^{\alpha-1/2}\alpha\int_{\Gamma_1}\xi^{2\alpha-1} 
\eta^{-\alpha-1}
(\xi+\eta-1)^{-\alpha-1/2}
d\xi\wedge d\eta
\\
=&\int_0^1
(-1)^{\alpha}\dfrac{-2ix^{-4\alpha}dx}{2x^2-2x+1}
-2\int_0^{i\infty}\dfrac{(4u^2)^{\alpha}du}
{1-u^2}
\\
=&\int_0^1
(-1)^{\alpha}\dfrac{-2ix^{-4\alpha}dx}{2x^2-2x+1}
-(-1)^{\alpha}4^{\alpha}\dfrac{\pi i}
{\cos(\pi\alpha)}
\end{align*}
By analytic continumation, we have the following theoreum.
\begin{theorem}
Let $2\alpha \not\in \bold Z$. Then we have the following equalities.
\begin{enumerate} 
\item
\begin{align*}
&
\dfrac{\alpha}{(\alpha+1/2)(-2 \alpha+1/2)}
F(1, 1, 1/2;\alpha+3/2, -2 \alpha+3/2;1)
\\
=&
\int_{P(0,1)}\dfrac{2x^{-4\alpha}dx}{2x^2-2x+1}
+\dfrac{4^{\alpha}\pi}
{\cos(\pi\alpha)}
\end{align*}
\item
\begin{align*}
&
\dfrac{\alpha-1/2}{(-2\alpha+3/2)\alpha}F(1,1,1/2;
\alpha+1,-2\alpha+5/2;1)
 =
\int_{P(0,1)}\dfrac{2x^{-4\alpha+2}dx}{2x^2-2x+1}
+\dfrac{2^{2\alpha-1}\pi}
{\cos(\pi\alpha-\pi/2)}
\end{align*}
\end{enumerate}
\end{theorem}
The second statement is obtained by setting $\alpha\mapsto \alpha-1/2$.
\subsubsection{Hypergeometric identities for $X_{3m}, X_{4m}$}

Using Stokes' theorem,  the integral is computed using the boundary integral
on $\Gamma_{2}$.
By Proposition \ref{relative integral and hgf},
we have the following theorem.
\begin{theorem}
\begin{enumerate}
\item
Let $\alpha$ be a real number such that $-1/3<\alpha<1/2$.
\begin{align*}
&\dfrac{3\alpha}{(2\alpha+2/3)(-2 \al+1)}
F(1, 1, \al+1;2 \al+5/3, -2 \al+2;1)
\\
=&3\alpha\int_{\Gamma_1}\xi^{3\alpha-1} 
\eta^{-\alpha-2/3}
(\xi+\eta-1)^{-\alpha-1/3}
d\xi\wedge d\eta
\\
=&
\int_{P(0,1)}\dfrac{27 x^{-6\alpha+2}(3x^3+1)}{1+27x^6}
dx
-
\int_{P(0,1)}\dfrac{9x^{6\alpha+1}(3-x^3)}{27+x^6}dx
-\dfrac{3^{3\alpha}\pi}{\sin(\pi\alpha+\pi/3)}.
\end{align*}
Moreover the first line is equal to the third line if $3\alpha\not\in\bold Z$.

\item
Let $\alpha$ be a real number such that $-1/4<\alpha<1/2$.
\begin{align*}
&-\dfrac{\alpha}{(3\alpha+3/4)(-2\al+1)}
F(1, 1, 2 \al+1;3 \al+7/4, -2 \al+2;1)
\\
=&-\alpha
\int_{\Gamma_1}(\xi+\eta-1)^{-\alpha-1/4}
\xi^{4\alpha-1} 
\eta^{-\alpha-3/4}
d\xi\wedge d\eta
\\
=&-\int_{P(0,1)}x^{-8\alpha}
\dfrac{8(1+8x^4)x^3}{64x^8+1}dx- 
\int_{P(0,1)}x^{12\alpha} 
\dfrac{2x^2 (-8-4 x^4+x^8)}{64+x^{12}} dx 
\\
&+4^{3\alpha-1}\dfrac{\pi}
{\sqrt{2}\sin(\pi \alpha+\pi/4)}
\end{align*}
Moreover the first line is equal to the third line if $4\alpha\not\in\bold Z$.
\end{enumerate}
\end{theorem}
\begin{proof}
(1) We apply Skokes' formula for $X_{3m}$ and we have
\begin{align*}
&-3\alpha\int_{\Gamma_{2}}(\xi+\eta-1)^{3\alpha-1} 
\xi^{-\alpha-2/3}
\eta^{-\alpha-1/3}
d\xi\wedge d\eta
\\
=&-3(-1)^{2\alpha+2/3}\alpha\int_{\Gamma_1}\xi^{3\alpha-1} 
\eta^{-\alpha-2/3}
(\xi+\eta-1)^{-\alpha-1/3}
d\xi\wedge d\eta
\\
=&-(-1)^{2\alpha}\omega
\int_0^1\dfrac{27 x^{-6\alpha+2}(3x^3+1)}{1+27x^6}
dx
+(-1)^{2\alpha}\omega \int_0^1x^{6\alpha}
\dfrac{9x(3-x^3)dx}{27+x^6}
\\
&-(-1)^{2\alpha}\omega\dfrac{3^{3\alpha}\pi}{\sin(\pi\alpha+\pi/3)}.
\end{align*}
(2) is similar by applying Stokes' formula to $X_{4m}$.
\end{proof}
\section{Hypergeometric identities for $\Gamma_{3}$}

\subsection{Restriction of $\psi$ to the boundary for $\Gamma_{3}$}
In this section, we give a hypergeometric identities arising from
the integration on $\Gamma_{3}$ for $p=3,4$.

\subsubsection{Restriction to the boundary $B_y$}
On $X_{pm}$, we have
$$
\psi_{B_y}=
\begin{cases}
(-1)^{2\alpha}\dfrac{27\omega^2y^{-6\alpha+2}(-3y^3+1)}{1+27y^6}dy
&\quad p=3
\\ 
(-1)^{3\alpha}y^{-8\alpha}
\dfrac{(32+32 i)  (8 y^4-1) y^3}{64 y^8+1} dy
&\quad p=4
\end{cases}
$$

\subsubsection{Restriction to the boundary $B_{3}$}
We restrict $\psi$ to $B_{3}$ and change parameters of integrals on $X_{pm}$, 
\begin{align*}
\psi_{B_{3}}=&
\begin{cases}
-(-1)^{2\alpha}{y}^{6\alpha}
\dfrac{9\omega^2(9-{y}^3)dy}{27+{y}^6}
&\quad p=3
\\
-(-1)^{3\alpha}y^{12\alpha}
\dfrac{(8+8 i) (8 y^4-32+y^8)}{64+y^{12}}dy 
&\quad p=4
\end{cases}
\end{align*}

\subsubsection{Contribution from an exceptional divisor}

By setting $y-1=ux$, $u=-\bold e(-1/2p)v$ on $X_{pm}$,
($v \in \bold R_+$), 
the restriction of $\psi$ to the 
exceptional divisor  
$E_{3}\subset \Gamma_{3}$ is equal to
$$
\psi_{E_{3}}=
\begin{cases}
-(-1)^{2\alpha}\omega^2\dfrac{3(27v^3)^{\alpha} vdv}{v^3+1},
&\text{ on }X_{3m}
\\
-(-1)^{3\alpha}(4v)^{4\alpha}\dfrac{(-8-8 i) v^2 }{4 v^4+1}dv
&\text{ on }X_{4m}
\end{cases}
$$
Then the integrals are equal to the following.
$$
\int_0^{\infty}\psi_{E_{3}}=
\begin{cases}
-(-1)^{2\alpha}\omega^2\dfrac{3^{3\alpha}\pi}{\sin(\pi\alpha+2\pi/3)}
&\text{ on }X_{3m}
\\
-(-1)^{3\alpha}4^{3\alpha}\dfrac{(-1/2-1/2 i) 
\sqrt{2} \pi}{\cos(\pi \alpha+\pi/4)}
&\text{ on }X_{4m}
\end{cases}
$$

\subsection{Hypergeometric identities for $\Gamma_{3}$}

We similarly compute the integral on $\Gamma_{3}$ and
get the following theorem.
\begin{theorem}
\begin{enumerate}
\item
Let $\alpha$ be a real number such that $-1/6<\alpha<1/3$.
\begin{align*}
&
\dfrac{3\alpha}{(-2\alpha+1)(2 \al+1/3)}
F(1, 1, \al+1;-2 \al+2, 2 \al+4/3;1)
\\
=&3\alpha\int_{\Gamma_1}\eta^{3\alpha-1} 
(\xi+\eta-1)^{-\alpha-2/3}
\xi^{-\alpha-1/3}
d\xi\wedge d\eta
\\
=&
-\int_{P(0,1)}\dfrac{27y^{-6\alpha+2}(-3y^3+1)}{1+27y^6}dy
-\int_{P(0,1)}
\dfrac{9 y^{6\alpha}(9-y^3)}{27+y^6}dy
+\dfrac{3^{3\alpha}\pi}{\sin(\pi\alpha+2\pi/3)}
\end{align*}
Moreover the first line is equal to the third line if $3\alpha\not\in\bold Z$.
\item
Let $\alpha$ be a real number such that $-1/12<\alpha<1/4$.
 \begin{align*}
&-\dfrac{\alpha}{(-2\alpha+1)(3\alpha+1/4)}
F(1, 1, 2 \al+1;-2 \al+2, 3 \al+5/4;1)
\\
=&-\alpha
\int_{\Gamma_1} 
(\xi+\eta-1)^{-\alpha-3/4}
\xi^{-\alpha-1/4}\eta^{4\alpha-1}
d\xi\wedge d\eta
\\
=&\int_{P(0,1)}y^{-8\alpha}
\dfrac{8  (1-8 y^4) y^3}{64 y^8+1} dy
-\int_{P(0,1)}y^{12\alpha}
\dfrac{2 (8 y^4-32+y^8)}{64+y^{12}}dy 
\\
&-4^{3\alpha-1}\dfrac{\pi}
{\sqrt{2}\cos(\pi \alpha+\pi/4)}
\end{align*}
Moreover the first line is equal to the third line if $4\alpha\not\in\bold Z$.
\end{enumerate}
\end{theorem}



\appendix

\section{List of orbits for exceptional characters}

\noindent

In this section, we give a list of orbits of
$(\alpha_0,\alpha_1,\alpha_2,\alpha_3)$
under the action of multiplicative group $(\bold Z/m\bold Z)^{\times}$
satisfying the condition (\ref{Hodge cycles}).
The number of exceptional characters are written as $e_m$, the number of orbits are written as $o_m$.

\noindent
$m=12$, $e_m=8$, $o_m=2$
$$
\{(1, 4, 9, 10), (1, 6, 8, 9)\}
$$
$m=14$, $e_m=2$, $o_m=1$
$$
\{(1, 7, 9, 11)\}
$$
$m=15$, $e_m=8$, $o_m=1$
$$
\{(1, 6, 10, 13)\}
$$
$m=18$, $e_m=18$, $o_m=3$
$$
\{(1, 6, 14, 15), (1, 7, 12, 16), (1, 9, 12, 14)\}
$$
$m=20$, $e_m=26$, $o_m=4$
$$
\{(1, 4, 17, 18), (1, 6, 16, 17), (1, 9, 13, 17), (1, 10, 12, 17)\}
$$
$m=21$, $e_m=12$, $o_m=1$
$$
\{(1, 4, 18, 19)\}
$$
$m=24$, $e_m=38$, $o_m=8$
$$
\{(1, 6, 19, 22), (1, 8, 17, 22), (1, 8, 19, 20), (1, 11, 17, 19), $$
$$
(1, 12, 16, 19), (1, 12, 17, 18), (1, 13, 16, 18), (2, 9, 16, 21)\}
$$
$m=28$, $e_m=10$, $o_m=2$
$$
\{(1, 9, 21, 25), (1, 15, 18, 22)\}
$$
$m=30$, $e_m=98$, $o_m=15$
$$
\{(1, 4, 27, 28), (1, 7, 25, 27), (1, 8, 25, 26), (1, 10, 24, 25), (1, 11, 24, 24), 
$$
$$
(1, 12, 20, 27), (1, 12, 23, 24), (1, 15, 17, 27), (1, 15, 19, 25), (1, 15, 20, 24), 
$$
$$
(1, 16, 21, 22), (1, 17, 19, 23), (1, 18, 20, 21), (1, 19, 20, 20), (2, 15, 21, 22)\}
$$
$m=36$, $e_m=18$, $o_m=2$
$$
\{(1, 19, 24, 28), (2, 9, 28, 33)\}
$$
$m=40$, $e_m=16$, $o_m=2$
$$
\{(1, 21, 24, 34), (1, 21, 26, 32)\}
$$
$m=42$, $e_m=166$, $o_m=16$
$$
\{(1, 6, 37, 40), (1, 6, 38, 39), (1, 8, 37, 38), (1, 12, 33, 38), (1, 13, 32, 38), 
$$
$$
(1, 15, 31, 37), (1, 15, 32, 36), (1, 16, 30, 37), (1, 16, 33, 34), (1, 18, 32, 33), 
$$
$$
(1, 19, 24, 40), (1, 21, 24, 38), (1, 21, 25, 37), (1, 21, 29, 33), (1, 21, 30, 32), (1, 24, 29, 30)\}
$$
$m=48$, $e_m=16$, $o_m=2$
$$
\{(1, 25, 32, 38), (1, 25, 34, 36)\}
$$
$m=60$, $e_m=204$, $o_m=23$
$$
\{(1, 12, 49, 58), (1, 15, 49, 55), (1, 17, 49, 53), (1, 20, 41, 58), (1, 20, 49, 50), 
$$
$$
(1, 23, 47, 49), (1, 24, 41, 54), (1, 24, 46, 49), (1, 25, 45, 49), (1, 27, 41, 51), 
$$
$$
(1, 29, 41, 49), (1, 30, 40, 49), (1, 30, 41, 48), (1, 31, 34, 54), (1, 31, 38, 50), 
$$
$$
(1, 31, 40, 48), (1, 31, 42, 46), (1, 36, 41, 42), (2, 15, 48, 55), (2, 21, 40, 57), 
$$
$$
(2, 21, 46, 51), (2, 25, 38, 55), (3, 32, 33, 52)\}
$$
$m=66$, $e_m=30$, $o_m=2$
$$
\{(1, 25, 44, 62), (2, 39, 45, 46)\}
$$
$m=72$, $e_m=12$, $o_m=1$
$$
\{(3, 16, 57, 68)\}
$$
$m=78$, $e_m=32$, $o_m=2$
$$
\{(1, 32, 61, 62), (1, 39, 55, 61)\}
$$
$m=84$, $e_m=66$, $o_m=6$
$$
\{(1, 29, 63, 75), (1, 43, 48, 76), (1, 43, 50, 74), (1, 43, 58, 66), (1, 43, 60, 64), (2, 
33, 58, 75)\}
$$
$m=90$, $e_m=24$, $o_m=1$
$$
\{(3, 20, 72, 85)\}
$$
$m=120$, $e_m=72$, $o_m=5$
$$
\{(1, 49, 83, 107), (1, 61, 80, 98), (1, 61, 82, 96), (2, 25, 98, 115), (4, 25, 96, 115)\}
$$
$m=156$, $e_m=24$, $o_m=1$
$$
\{(1, 79, 110, 122)\}
$$
$m=180$, $e_m=24$, $o_m=1$
$$
\{(3, 40, 147, 170)\}
$$

\vskip 0.3in

\end{document}